\documentclass[12pt]{amsart}

\usepackage{amssymb, amsmath, amsthm}
\usepackage[margin=1in]{geometry}
\usepackage[dvipdfmx]{graphicx,xcolor} 
\usepackage{tikz}

\allowdisplaybreaks

\numberwithin{equation}{section}

\theoremstyle{plain}
\newtheorem{thm}{Theorem}[section]
\newtheorem{prop}[thm]{Proposition}
\newtheorem{lem}[thm]{Lemma}

\newtheorem*{referthmA}{Theorem A}
\newtheorem*{referthmB}{Theorem B}

\theoremstyle{definition}
\newtheorem{defn}[thm]{Definition}
\newtheorem{rem}[thm]{Remark}

\newtheorem{notation}[thm]{Notation}

\newcommand{\ichi}{\mathbf{1}}

\newcommand{\N}{\mathbb{N}}
\newcommand{\R}{\mathbb{R}}

\newcommand{\Z}{\mathbb{Z}}

\newcommand{\calM}{\mathcal{M}}

\newcommand{\calS}{\mathcal{S}}

\newcommand{\supp}{\mathrm{supp}\, }

\newcommand{\diam}{\, \mathrm{diam}\, }

\newcommand{\RomI}{\mathrm{I}}
\newcommand{\II}{\mathrm{I}\hspace{-0.5pt}\mathrm{I}}
\newcommand{\III}{\mathrm{I}\hspace{-0.5pt}\mathrm{I}\hspace{-0.5pt}\mathrm{I}}

\begin{document}

\title[Bilinear Fourier multipliers]
{On some bilinear Fourier multipliers with oscillating factors, II} 

\author[T. Kato]{Tomoya Kato}
\author[A. Miyachi]{Akihiko Miyachi}
\author[N. Shida]{Naoto Shida}
\author[N. Tomita]{Naohito Tomita}

\address[T. Kato]
{Faculty of Engineering, 
Gifu University, 
Gifu 501-1193, Japan}

\address[A. Miyachi]
{Department of Mathematics, 
Tokyo Woman's Christian University, 
Zempukuji, Suginami-ku, Tokyo 167-8585, Japan}

\address[N. Shida]
{Graduate School of Mathematics, 
Nagoya University, 
Furocho, Chikusa-ku, Nagoya 464-8602, Japan}

\address[N. Tomita]
{Department of Mathematics, 
Graduate School of Science, Osaka University, 
Toyonaka, Osaka 560-0043, Japan}

\email[T. Kato]{kato.tomoya.s3@f.gifu-u.ac.jp}
\email[A. Miyachi]{miyachi@lab.twcu.ac.jp}
\email[N. Shida]{naoto.shida.c3@math.nagoya-u.ac.jp}
\email[N. Tomita]{tomita@math.sci.osaka-u.ac.jp}

\date{\today}

\thanks{This work was supported by JSPS KAKENHI, 
Grant Numbers 
23K12995 (Kato), 23K20223 (Miyachi), 
23KJ1053 (Shida), and 20K03700 (Tomita).}

\maketitle

\begin{abstract}
For $s > 0$, $s \neq 1$,
bilinear Fourier multipliers of the form 
$e^{i (|\xi|^s + |\eta|^s+ |\xi + \eta|^s)} \sigma (\xi, \eta)$ 
are considered, 
where $\sigma(\xi, \eta)$ belongs to the H\"ormander class $S^{m}_{1, 0}(\R^{2n})$.
A criterion for $m$ to ensure 
the $L^{\infty}\times L^{\infty} \to L^\infty$, 
$L^{1} \times L^{\infty} \to L^{1}$, and 
$L^{\infty} \times L^{1} \to L^{1}$
boundedness of the corresponding bilinear operators
is given.
\end{abstract}

\section{Introduction}

Throughout this paper, 
the letter $n$ denotes a positive integer.

This paper is a continuation of the authors' paper \cite{KMST-preprint}.
We are interested in bilinear Fourier multipliers of the form
\begin{align} \label{3-phase}
e^{i(\phi_1(\xi) + \phi_2(\eta) + \phi_3(\xi+\eta))} \sigma(\xi, \eta),
\quad
\sigma \in S^{m}_{1, 0}(\R^{2n}),
\end{align}
where $\phi_1, \phi_2$ and $\phi_3$ are real-valued smooth functions on $\R^n \setminus \{0\}$ (see Definition \ref{defn_Hormander} for the class $S^m_{1, 0}$).
In \cite{KMST-preprint},
the bilinear multipliers \eqref{3-phase} with $\phi_1 = \phi_2 = \phi_3 =|\cdot|$ 
are considered.
In the present paper, 
we focus on the bilinear Fourier multipliers
of the form \eqref{3-phase} with $\phi_1 = \phi_2 = \phi_3 = |\cdot|^{s}$,
$s > 0$, $s \neq 1$.
In what follows, we always assume that $s > 0$ and $s \neq 1$.
For related results in the case $s =1$, 
see Introduction of \cite{KMST-preprint}.

We briefly recall a result for linear Fourier multipliers.
Let $\theta = \theta(\xi)$ be a measurable function 
on $\R^n$ satisfying $|\theta(\xi)| \le c(1+ |\xi|)^{L}$
for some $c > 0$ and $L \in \R$.
Then the Fourier multiplier operator $\theta(D)$ is defined by
\begin{align*}
\theta(D)f(x)
=
\frac{1}{(2\pi)^n}
\int_{\R^n}
e^{ix \cdot \xi}
\theta(\xi)
\widehat{f}(\xi)
\,
d\xi
\end{align*}
for $f \in \calS(\R^n)$, where $\widehat{f}$ denotes the Fourier transform
of $f$.
We call the function $\theta$ multiplier.
For function spaces on $\R^n$ $(X, \|\cdot\|_X)$ and $(Y, \|\cdot\|_Y)$,
we say that $\theta$ is a {\it Fourier multiplier} for $X \to Y$
if there exists a positive constant $A$ such that 
\begin{align*}
\big\|
\theta(D) f
\big\|_{Y}
\le
A
\|f\|_{X}
\quad
\text{for all}
\quad
f \in \calS \cap X,
\end{align*}
and write $\theta(\xi) \in \calM(X \to Y)$.

We consider the following class.
\begin{defn} \label{defn_Hormander}
Let $d \in \N$ and $m \in \R$.
The class $S^m_{1, 0}(\R^d)$ is defined to be 
the set of all $\sigma = \sigma(\zeta) \in C^{\infty}(\R^d)$
satisfying
\begin{align*}
\big|
\partial_{\zeta}^{\alpha}
\sigma(\zeta)
\big|
\le
C_{\alpha}
(1+|\zeta|)^{m-|\alpha|}
\end{align*}
for all multi-indices $\alpha \in (\N_0)^d = \{0, 1, 2, \dots \}^d$.
\end{defn}

Then the following is known.
\begin{referthmA}
Let $s \neq 1$ and $1 \le p \le \infty$.
Then for every $\sigma \in S^m_{1, 0}(\R^n)$ with
$m = -ns|1/p -1/2|$ the function 
$e^{i|\xi|^s} \sigma(\xi)$ is a
{\it Fourier multiplier} for $H^p \to L^p$ if $1 \le p < \infty$,
and for $L^\infty \to BMO$ if $p = \infty$. 
\end{referthmA}

Theorem A is given by
Sj\"olin \cite{Sjolin-MathZ} and Miyachi \cite{Miyachi-multiplier}
for the case $s < 1$,
and by Miyachi \cite{Miyachi-multiplier}
for the case $s > 1$.
It is also proved in Sj\"ostrand \cite[Lemma 2.1 and Theorem 4.1]{Sjostrand} and Miyachi \cite{Miyachi-multiplier} 
that the number $-ns|1/p -1/2|$ is sharp if $s \neq 1$.

Next, we recall the definition of bilinear Fourier multipliers.
Let $\sigma = \sigma(\xi, \eta)$ be a measurable function 
on $\R^n \times \R^n$ 
satisfying $|\sigma(\xi, \eta)| \le c(1+ |\xi| + |\eta|)^{L}$
for some $c > 0$ and $L \in \R$.
Then the bilinear Fourier multiplier operator $T_{\sigma}$ is defined by
\begin{align*}
T_{\sigma}(f, g)(x)
=
\frac{1}{(2\pi)^{2n}}
\iint_{\R^n \times \R^n}
e^{ix \cdot (\xi + \eta)}
\sigma(\xi, \eta)
\widehat{f}(\xi)
\widehat{g}(\eta)
\,
d\xi d\eta
\end{align*}
for $x \in \R^n$ and $f, g \in \calS(\R^n)$.
Let $(X, \|\cdot\|_X)$, $(Y, \|\cdot\|_Y)$, and $(Z, \|\cdot\|_Z)$ 
be function spaces on $\R^n$.
We say that $\sigma$ is a {\it bilinear Fourier multiplier} for $X \times Y \to Z$
if there exists a positive constant $A$ such that 
\begin{align*}
\big\|
T_{\sigma}(f, g)
\big\|_{Z}
\le
A
\|f\|_{X}
\|g\|_{Y}
\quad
\text{for all}
\quad
f \in \calS \cap X
\;\;
\text{and}
\;\;
g \in \calS \cap Y,
\end{align*}
and write $\sigma(\xi, \eta) \in \calM(X \times Y \to Z)$.

The bilinear Fourier multipliers 
of the form \eqref{3-phase}
was studied by 
Bernicot and Germain \cite{BG-AM},
and
by
Bergfeldt, Rodr\'iguez-Lop\'ez, Rule and Staubach \cite{BRRS-TAMS}.
We say that a function $\phi$ is a {\it phase function of order $s$},
if $\phi$ is a real-valued smooth function on $\R^n \setminus \{0\}$
and satisfies the estimate 
\[
\big|
\partial^{\alpha}_{\xi} \phi(\xi)
\big|
\le
C_{\alpha} |\xi|^{s-|\alpha|},
\quad
\xi \neq 0,
\]
for all $\alpha \in (\N_0)^n$.
A typical example is the function $\phi = |\cdot|^s$.
The following is known.
\begin{referthmB}[{\cite[Theorem 1.4 and Remark 1.5]{BRRS-TAMS}}]
Let $s \neq 1$.
Let $\phi_1, \phi_2$ and $\phi_3$ be phase functions of order $s$.
Suppose that $1 \le p, q, r \le \infty$ and $1/p + 1/q = 1/r$.
Define $m_s(p, q) \in \R$ by
\[
m_s(p, q)
=
\begin{cases}
-ns\big(|1/p -1/2| + |1/q -1/2| \big)
&
\text{if} 
\;\;
\phi_3 \equiv 0,
\\
-ns\big(|1/p -1/2| + |1/q -1/2| + |1/r -1/2| \big)
&
\text{if} 
\;\;
\phi_3 \not\equiv 0.
\end{cases}
\] 
Then for every $\sigma \in S^m_{1, 0}(\R^{2n})$ with $m = m_s(p, q)$,
the function 
$e^{i(\phi_1(\xi) + \phi_2(\eta) + \phi_3(\xi + \eta))} \sigma(\xi, \eta)$
is a bilinear Fourier multiplier for $H^p \times H^q \to L^r$,
where $L^r$ should be replaced by $BMO$ when $r = \infty$.
\end{referthmB}

In the case $\phi_3 \equiv 0$,
it is proved in \cite[Section 3]{BRRS-TAMS} and \cite[Theorem 1.4]{KMST-arXiv} 
that the condition $m = m_s(p, q) = -ns(|1/p -1/2| + |1/q -1/2|)$ 
is sharp in the ranges
$1 \le p, q \le 2$ and $2 \le p, q \le \infty$.
Kato--Miyachi--Shida--Tomita \cite{KMST-arXiv} 
recently improved Theorem B in the special case $\phi_1 = \phi_2 = |\cdot|^s$ and $\phi_3 \equiv 0$.
In this case, it is proved that 
the condition $m = m_s(p, q)$ given in Theorem B can be relaxed 
for $p < 2 < q $ or $q < 2 < p$.

We shall be interested in the case $\phi_3 \not\equiv 0$.
As proved in \cite[Section 3]{BRRS-TAMS},
using Plancherel's theorem, 
we see that the number $m_s(p, q)$ is sharp when $2 \le p, q \le \infty$ and $r = 2$.
To the best of the authors' knowledge, the sharpness has not been known 
in other cases so far.

In the present paper, we focus on the typical case 
where $\phi_1 = \phi_2 = \phi_3 = |\cdot|^s$, $s \neq 1$.
Then the purpose of this paper is to relax the condition $m = m_s(p, q)$
of Theorem B for the case $(p, q) = (1, \infty), (\infty, 1), (\infty, \infty)$.
The main result of the present article reads as follows.

\begin{thm} \label{mainthm}
Let $s \neq 1$.
Then for every $\sigma \in S^m_{1, 0}(\R^{2n})$
with
\[
m < m_{s} =
\begin{cases}
-{ns}/{2} -ns(1-s), & \textrm{if} \;\; 0 < s < 1, \\
-{ns}/{2}, & \textrm{if} \;\; 1 < s \le 2, \\
-n(s-1), & \textrm{if} \;\; 2 < s < \infty,
\end{cases}
\]
the function $e^{i(|\xi|^s + |\eta|^s + |\xi + \eta|^s)} \sigma(\xi, \eta)$
is a bilinear Fourier multiplier for $L^\infty \times L^\infty \to L^\infty$,
$L^1 \times L^\infty \to L^1$, and $L^\infty \times L^1 \to L^1$.
\end{thm}

Moreover, in the case $s > 1$, we reveal that
Theorem \ref{mainthm} cannot be 
extended to the case $m > m_s$.
In this sense, the order $m_s$ is optimal
for $s > 1$.
See Proposition \ref{prop_s>1} in Section \ref{sharpness}.

\begin{rem}
(1)
For every $n \ge 1$ and $s \neq 1$,
the number $m_s$ defined in Theorem \ref{mainthm}
is strictly bigger than the number given by
Theorem B;
the latter is $m_s(1, \infty) = m_s(\infty, 1) = m_s(\infty, \infty) = -3ns/2$.

(2)
Unfortunately, the authors of the present paper 
do not know what happens in the case $m = m_s$.

(3) In the case $s < 1$, the present authors 
do not know whether the order $m_s$ of 
Theorem \ref{mainthm} is optimal or not.
However, we shall prove a result which shows
that  the methods of the present paper cannot 
be extended to the case $m > m_s$.
See Proposition \ref{prop_s<1} in Section \ref{sharpness}.
\end{rem}

The rest of this paper is organized as follows.
In Section \ref{Estimateforkernel}, we prepare several estimates for the function of the form
$(e^{i|\xi|^s} \theta(2^{-j}\xi))^{\vee}$.
In Section \ref{proof_of_mainthm}, 
we introduce the way to reduce the proof of Theorem \ref{mainthm},
and give a proof of Theorem \ref{mainthm}.
In Section \ref{sharpness}, we consider the sharpness of Theorem \ref{mainthm}.

We end this section by preparing some notations 
used in the present paper.

\begin{notation}\label{notation}
We denote by $\N$ and $\N_0$ the set of positive integers and
the set of nonnegative integers, respectively.
We write the ball on $\R^n$ with center $x$ and radius $r$
as
$B(x, r) = \{y \in \R^n \mid |y-x| < r\}$.

The Fourier transform and 
the inverse Fourier transform 
are defined by 
\begin{align*}&
\widehat{f}(\xi)
=
\int_{\R^n}
e^{-i \xi \cdot x}
f(x)\, dx
\quad\textrm{and}\quad
(g)^{\vee}(x)
=
\frac{1}{(2\pi)^n} 
\int_{\R^n}
e^{i \xi \cdot x}
g(\xi)\, d\xi. 
\end{align*}

We prepare the dyadic partition of unity on $\R^n$. 
Take functions $\varphi, \psi \in C^\infty_0 (\R^n)$ 
such that 
$\supp \varphi \subset 
\{ |\xi| \le 2 \}$,
$\varphi = 1$ on 
$\{ |\xi| \le 1 \}$,
$\supp \psi \subset 
\{ 1/2 \leq |\xi| \leq 2 \}$,
and
$\varphi+\sum_{j\in\N} \psi ( 2^{-j} \cdot ) = 1$.
We will write
$\psi_{0} = \varphi$,
$\psi_j = \psi ( 2^{-j} \cdot )$ 
for $j \in \N$, 
and 
$\varphi_j = \varphi ( 2^{-j} \cdot )$ 
for $j \in \N_{0}$.
Then, we see that 
$\psi_0 = \varphi_0 = \varphi$ and
\begin{align*}
\sum_{j=0}^{k} \psi_j = \varphi_k,  
\quad k\in \N_0. 
\end{align*}
We define $\zeta = 1 - \varphi$.
Then $\zeta \in C^\infty(\R^n)$, 
$\partial^{\alpha} \zeta 
\in C^{\infty}_{0} (\R^n)$
for $|\alpha|\ge 1$,
$\zeta = \sum_{j \in \N} \psi_{j}$,
$\zeta =0$ on $\{ |\xi|\le 1 \}$,
and
$\zeta =1$ on $\{ |\xi|\ge 2 \}$.

For a smooth function $\theta$ on $\R^n$ and 
for $N \in \N_{0}$, we write 
$\|\theta\|_{C^N}=
\max_{|\alpha|\le N} 
\sup_{\xi} \big| \partial_{\xi}^{\alpha} \theta (\xi) \big|$. 

For $0 < p \le \infty$,
$H^p$ denotes the usual Hardy space,
and $h^p$ denotes the local Hardy space.
$BMO$ denotes the space of bounded mean oscillation
and $bmo$ denotes its local version.
It is known that $BMO$ is the dual space of $H^1$,
and that $bmo$ is the dual space of $h^1$.
See Stein \cite[Chapters III and IV]{Steinbook}
for $H^p$ and $BMO$,
and Goldberg \cite{Goldberg} 
for $h^p$ and $bmo$.

\end{notation}

\section{Estimates for kernels} \label{Estimateforkernel}

In this section, we shall collect several lemmas
concerning the function $(e^{i|\xi|^s} \theta(2^{-j} \xi))^{\vee}$
with $\theta \in C^\infty_0(\R^n)$.

\begin{prop}[{\cite[Proposition 3.1]{KMST-arXiv}}] \label{prop_psi_KMST1}
Let $s \neq 1$. Suppose that $\psi \in \calS(\R^n)$ satisfies
$\supp \psi \subset \{ 1/2 \le |\xi| \le 2 \}$.
Then, for any $N_{1}, N_{2}, N_{3} \ge 0$,
there exist $c>0$ and $M \in \N$
such that
\begin{equation*} 
\Big| \Big( 
e^{ i |\xi|^{s} } 
\psi (2^{-j} \xi ) 
\Big)^{\vee} (x)
\Big|
\le c\,
\| \psi \|_{ C^M }
\times
\begin{cases}
{2^{-jN_1}} ,
	& \text{if}\;\; {2^{j (1-s)} |x| < a}, \\
{2^{j(n-\frac{ns}{2})}} ,
	& \text{if}\;\; {a\le 2^{j (1-s)} |x| \le b}, \\
{2^{-jN_2}|x|^{-N_3}} ,
	& \text{if}\;\; {2^{j (1-s)} |x| >b} ,
\end{cases}
\end{equation*}
for $j \in \N_{0}$,
where
$a = s 4^{ -|1-s| }$
and 
$b = s 4^{ |1-s| }$.

If in addition $\psi (\xi)\neq 0$ for $2/3 \le |\xi| \le 3/2$, then 
there exist $c^{\prime} >0$ and 
$j_0 \in \N$ such that  
\begin{equation} \label{Kjxradial}
\begin{split}
&
\frac{1}{c^{\prime}}
2^{j(n - \frac{ns}{2})}\le 
\Big| \Big( 
e^{ i |\xi|^{s} } 
\psi (2^{-j} \xi ) 
\Big)^{\vee} (x)
\Big|
\le 
c^{\prime}
2^{j(n - \frac{ns}{2})}
\quad
\text{if}
\;\;
a^{\prime} \le 2^{j(1-s)}|x| \le b^{\prime} 
\;\; 
\text{and}
\;\;
j>j_0, 
\end{split}
\end{equation}
where
$a^{\prime} = s (3/2)^{ -|1-s| }$
and 
$b^{\prime} = s (3/2)^{ |1-s| }$.

\end{prop}

\begin{prop}[{\cite[Estimate (4.8)]{KMST-arXiv}}] \label{prop_phi_KMST1}
Suppose that $\theta \in \calS(\R^n)$ satisfies
$\supp \theta \subset \{ |\xi| \le a \}$ for some $a > 0$,
and the function $\varphi$ is as in Notation \ref{notation}.
Then there exist $c > 0$ and $M \in \N$ such that
\begin{align*}
\Big| \Big( 
e^{ i |\xi|^{s} } 
\varphi(\xi)
\theta (2^{-j} \xi ) 
\Big)^{\vee} (x)
\Big|
\le c\,
\| \theta \|_{ C^M } \,
(1+|x|)^{ -(n+s) } ,
\quad
x \in \R^n,
\quad
j \in \N_0.
\end{align*} 
\end{prop}

\begin{prop}[{\cite[Corollary 3.2]{KMST-arXiv}}] \label{prop_zeta_KMST1}
Suppose that $\theta \in \calS(\R^n)$ satisfies
$\supp \theta \subset \{ |\xi| \le a \}$ for some $a >0$,
and the function $\zeta$ is as in Notation \ref{notation}.
Then the following hold.
\begin{enumerate}
\item \label{Kj_phi_s<1}
Let $0< s < 1$ and $N\ge 0$.
Then,
there exist $c>0$ and $M \in \N$
such that
\begin{equation*}
\Big| \Big( 
e^{ i |\xi|^{s} } 
\zeta(\xi)
\theta (2^{-j} \xi ) 
\Big)^{\vee} (x)
\Big|
\le c\,
\| \theta \|_{ C^M }
\times
\begin{cases}
|x|^{- \frac{n}{2} - \frac{n}{2(1-s)} } ,
	& \text{if}\;\; |x| \le 1, \\
|x|^{ -N } ,
	& \text{if}\;\; |x| > 1 ,
\end{cases}
\end{equation*}
for all $j \in \N_{0}$.

\item \label{Kj_phi_s>1}
Let $1 < s < \infty$and $N\ge 0$.
Then,
there exist $C = C(a, s) > 0$, $c>0$ and $M \in \N$
such that
\begin{equation*} 
\Big| \Big( 
e^{ i |\xi|^{s} } 
\zeta(\xi)
\theta (2^{-j} \xi ) 
\Big)^{\vee} (x)
\Big|
\le c\,
\| \theta \|_{ C^M }
\times
\begin{cases}
\big( 1+|x| \big)^{- \frac{n}{2} + \frac{n}{2(s-1)} }
	& \text{for all}\;\; x \in \R^n, \\ 
	|x|^{ -N } ,
	& \text{if}\;\; |x| > C \, 2^{j(s-1)} ,
\end{cases}
\end{equation*}
for all $j \in \N_{0}$.
\end{enumerate}
\end{prop}

As a corollary of these propositions, 
we have the following.
 
\begin{lem} \label{kernel_sneq1}
Let $s \neq 1$. 
Suppose that $\theta \in \calS(\R^n)$ satisfies
$\supp \theta \subset \{ |\xi| \le a \}$ for some $a>0$.

\begin{enumerate}

\item 
If $s < 1$, then 
there exist $c>0$ and $M \in \N$
such that
\begin{equation*} 
\Big| \Big( 
e^{ i |\xi|^{s} } 
\theta (2^{-j} \xi ) 
\Big)^{\vee} (x)
\Big|
\le c\,
\| \theta \|_{ C^M }
\times
\begin{cases}
|x|^{- \frac{n}{2} - \frac{n}{2(1-s)} } ,
	& \text{if}\;\; |x| \le 1, \\
|x|^{ -(n+s) } ,
	& \text{if}\;\; |x| > 1,
\end{cases}
\end{equation*}
for all $j \in \N$.

\item
If $s > 1$, then there exist $C = C(a, s) > 0$, $c>0$ and $M \in \N$
such that
\begin{equation*} 
\Big| \Big( 
e^{ i |\xi|^{s} } 
\theta (2^{-j} \xi ) 
\Big)^{\vee} (x)
\Big|
\le c\,
\| \theta \|_{ C^M }
\times
\begin{cases}
\big( 1+|x| \big)^{- \frac{n}{2} + \frac{n}{2(s-1)} }
	& \text{for all}\;\; x \in \R^n , \\
|x|^{ -(n+s) } ,
	& \text{if}\;\; |x| > C \, 2^{j(s-1)} ,
\end{cases}
\end{equation*}
for all $j \in \N$.
\end{enumerate}

\end{lem}

\begin{proof}
Since $\zeta + \varphi = 1$, 
the claims immediately follow from
Propositions \ref{prop_phi_KMST1} and \ref{prop_zeta_KMST1}.
\end{proof}

\begin{lem} \label{kernel_sneq1_norm}
Let $s \neq 1$. 
Suppose that $\theta \in \calS(\R^n)$ satisfies
$\supp \theta \subset \{ |\xi| \le a \}$ for some $a>0$.
Then there exist $c>0$ and $M \in \N$
such that
\begin{equation*} 
\Big\| \Big( 
e^{ i |\xi|^{s} } 
\theta (2^{-j} \xi ) 
\Big)^{\vee}
\Big\|_{L^1}
\le c\,
\| \theta \|_{ C^M }
2^{j\frac{ns}{2}}
\end{equation*}
for all $j \in \N$.
Moreover, if $s < 1$, then there exist 
$c>0$ and $M \in \N$
such that
\begin{equation*} 
\Big\| \Big( 
e^{ i |\xi|^{s} } 
\theta (2^{-j} \xi ) 
\Big)^{\vee}
\Big\|_{L^\infty}
\le c\,
\| \theta \|_{ C^M }
2^{j(n - \frac{ns}{2})}
\end{equation*}
for all $j \in \N$.
\end{lem}

\begin{proof}
By Proposition \ref{prop_phi_KMST1}, we have
\begin{align*}
\Big\|
\Big(
e^{ i |\xi|^{s} } \varphi(\xi)\theta (2^{-j} \xi )
\Big)^{\vee}
\Big\|_{L^p} 
\le 
c \|\theta\|_{C^M},
\quad
1 \le p \le \infty.
\end{align*}
Hence it is sufficient to consider the $L^1$ or $L^\infty$ estimate 
of the function $(e^{i|\xi|^s} \zeta(\xi) \theta(2^{-j} \xi))^{\vee}$.

Let $j_0$ be an integer such that $2^{j_0} \ge a$.
We recall that the function $\zeta$ can be written as 
$\zeta = \sum_{k =1}^{j+j_0} \psi(2^{-k} \cdot)$ on $\supp \theta(2^{-j} \cdot)$.
By choosing $N_1$, $N_2$, and $N_3$ sufficiently large,
Proposition \ref{prop_psi_KMST1} yields that 
\begin{align*}
&
\Big\|
\Big( 
e^{ i |\xi|^{s} } 
\psi(2^{-k} \xi)
\theta (2^{-j} \xi ) 
\Big)^{\vee}
\Big\|_{L^1}
\le
c
\|\psi(\cdot)\theta(2^{k-j} \cdot)\|_{C^M}
2^{k \frac{ns}{2}}.
\end{align*}
Combining this with a simple estimate $\|\psi(\cdot) \theta(2^{k-j} \cdot)\|_{C^M} \lesssim \|\theta\|_{C^M}$ for $1 \le k \le j+j_0$, 
and summing over $1 \le k \le j+j_0$,
we have
\begin{align*} 
\Big\|
\Big(
e^{ i |\xi|^{s} } \zeta(\xi)\theta (2^{-j} \xi )
\Big)^{\vee}
\Big\|_{L^1} 
\le 
c \|\theta\|_{C^M}
2^{j\frac{ns}{2}},
\end{align*}
which completes the proof of the estimate for the $L^1$-norm.

Moreover, it follows from Proposition \ref{prop_psi_KMST1} that 
\begin{align*}
&
\Big\|
\Big( 
e^{ i |\xi|^{s} } 
\psi(2^{-k} \xi)
\theta (2^{-j} \xi ) 
\Big)^{\vee}
\Big\|_{L^\infty}
\le
c
\|\psi(\cdot)\theta(2^{k-j} \cdot)\|_{C^M}
2^{k(n - \frac{ns}{2})},
\end{align*}
which implies that
\begin{align*}
\Big\|
\Big(
e^{ i |\xi|^{s} } \zeta(\xi)\theta (2^{-j} \xi )
\Big)^{\vee}
\Big\|_{L^\infty}
\le
c
\|\psi(\cdot)\theta(2^{k-j} \cdot)\|_{C^M}\sum_{k=1}^{j+j_0}
2^{k(n - \frac{ns}{2})}.
\end{align*}
If $s < 1$, then the last sum can be estimated by $c 2^{j(n - \frac{ns}{2})}$,
and hence we obtain the desired estimate for $L^\infty$-norm.
We complete the proof of Lemma \ref{kernel_sneq1_norm}.
\end{proof}

\section{Proof of Theorem \ref{mainthm}} \label{proof_of_mainthm}

For a function $\tau = \tau(\xi, \eta)$,
the following equivalences are well known:
\begin{align*}
&
\tau(\xi, \eta) \in \calM(L^\infty \times L^\infty \to L^\infty)
\\
&
\iff
\tau(-\xi-\eta, \eta) \in \calM(L^1 \times L^\infty \to L^1)
\\
&
\iff
\tau(\xi, -\xi - \eta) \in \calM(L^\infty \times L^1 \to L^1).
\end{align*}
On the other hand, 
for the class $S^m_{1, 0}(\R^{2n})$,
the following equivalences hold:
\begin{align*}
\sigma(\xi, \eta) \in S^{m}_{1, 0}(\R^{2n})
\iff
\sigma(-\xi-\eta, \eta) \in S^{m}_{1, 0}(\R^{2n})
\iff
\sigma(\xi, -\xi - \eta) \in S^{m}_{1, 0}(\R^{2n}).
\end{align*}
Hence, 
the claims of Theorem \ref{mainthm}
regarding the boundedness
of
$L^\infty \times L^\infty \to L^\infty$,
$L^1 \times L^\infty \to L^1$,
and
$L^\infty \times L^1 \to L^1$
are equivalent to each other,
and to prove Theorem \ref{mainthm},
it is sufficient to prove one of them.
In this section, we shall prove that 
if $m < m_s$, then
\begin{align} \label{claim_general}
e^{i(|\xi|^s + |\eta|^s + |\xi + \eta|^s)}
\sigma(\xi, \eta)
\in
\calM(L^\infty \times L^\infty \to L^\infty)
\;\;
\text{for all}
\;\;
\sigma \in S^m_{1, 0}(\R^{2n}).
\end{align}

\subsection{Reduction of the proof} \label{red_of_proof}

In this subsection, we recall the argument
to reduce the proof of 
the claim \eqref{claim_general}
to the case of $\sigma$ of product forms.
The basic idea of this argument goes back 
to Coifman--Meyer \cite{CM-Asterisque};
the argument can also be found in \cite[Section 7.5.2]{Grafakos-modern}.

Firstly, it is sufficient to prove that 
there exist $c > 0$ and $M \in \N$,
such that the estimate
\begin{align} \label{est_red}
\big\|
e^{i(|\xi+\eta|^s + |\xi|^s + |\eta|^s)}
\widetilde{\sigma}(\xi, \eta)
\big\|_{\calM(L^\infty \times L^\infty \to L^\infty)}
\le
c
\|\theta_1\|_{C^M}
\|\theta_2\|_{C^M}
\end{align}
holds with $\widetilde{\sigma}$ either of the form
\begin{align} \label{sigma_case1}
\widetilde{\sigma}(\xi, \eta) 
= 
c_0
\theta_1(\xi)\theta_2(\eta)
\end{align}
or
\begin{align} \label{sigma_case2}
\widetilde{\sigma}(\xi, \eta) 
=
\sum_{j =1}^{\infty}
c_j
\theta_1(2^{-j} \xi) \theta_2(2^{-j} \eta),
\end{align}
where $\{c_j\}_{j \in\N_0}$ is a sequence of
complex numbers satisfying $|c_j| \le 2^{jm}$,
and $\theta_1, \theta_2 \in C^\infty_0(\R^n)$.
To confirm this, we use the partition of unity given in Notation \ref{notation}
and decompose $\sigma$ as
\begin{align*}
&
\sigma(\xi, \eta)
=
\sum_{j, k \in \N_0}
\sigma(\xi, \eta)
\psi_j(\xi) \psi_k(\eta)
=
\Big(
\sum_{j - k \ge 3}
+
\sum_{|j - k| \le 2}
+
\sum_{j-k \le -3}
\Big)
\sigma(\xi, \eta)
\psi_j(\xi) \psi_k(\eta)
\\
&
=
\sum_{j =3}^{\infty}
\sigma(\xi, \eta)
\psi_{j}(\xi) \varphi_{j-3}(\eta)
+
\sum_{\ell = -2}^{2}
\sum_{j = \max \{0, \ell\}}^{\infty}
\sigma(\xi, \eta)
\psi_{j}(\xi) \psi_{j- \ell}(\eta)
+
\sum_{k= 3}^{\infty}
\sigma(\xi, \eta)
\varphi_{k-3}(\xi)
\psi_{k}(\eta) 
\\
&
= :
\sigma_{\RomI}(\xi, \eta)
+
\sigma_{\II}(\xi, \eta)
+
\sigma_{\III}(\xi, \eta).
\end{align*}
We consider the multiplier $\sigma_{\RomI}$.
Since it holds that for $\sigma \in S^m_{1, 0}(\R^{2n})$
\begin{align*}
\big|
\partial^\alpha_{\xi}
\partial^\beta_{\eta}
\sigma(\xi, \eta)
\big|
\le
C_{\alpha, \beta}
2^{j(m -|\alpha| -|\beta|)}
\end{align*}
on a neighborhood of $\supp \psi_j \times \supp \varphi_{j-3}$,
using Fourier series expansion of period $(2 \pi 2^{j} \Z^n) \times (2\pi 2^{j-3} \Z^n)$, we can write $\sigma_{\RomI}$ as
\begin{align*}
\sigma_{\RomI}(\xi, \eta)
=
\sum_{j =3}^\infty
\sum_{a, b \in \Z^n}
c_{\RomI, j}(a, b)
e^{i a \cdot 2^{-j} \xi}
e^{i b \cdot 2^{-j+3} \eta}
\psi_{j}(\xi) \varphi_{j-3}(\eta),
\end{align*}
where the coefficient $c_{\RomI, j}(a, b)$ satisfy 
that for each $N >0$
\begin{align*}
|c_{\RomI, j}(a, b)|
\lesssim
2^{jm}
(1+|a|)^{-N}
(1+|b|)^{-N}.
\end{align*}
Hence, 
once we show that \eqref{est_red} holds with $\widetilde{\sigma}$ of the form \eqref{sigma_case2}, then we see that 
the estimate
\begin{align*}
\big\|
e^{i(|\xi+\eta|^s + |\xi|^s + |\eta|^s)}
\sigma_{\RomI}(\xi, \eta)
\big\|_{\calM(L^\infty \times L^\infty \to L^\infty)}
\lesssim
\sum_{a, b \in \Z^n}
(1+|a|)^{-N+M} 
(1+|b|)^{-N+M}
\end{align*}
holds for any $N > 0$. Therefore,
it follows  from choosing $N > M+n$
that the function $e^{i(|\xi+\eta|^s + |\xi|^s + |\eta|^s)}
\sigma_{\RomI}(\xi, \eta)$ is a bilinear Fourier multiplier 
for $L^\infty \times L^\infty \to L^\infty$.
Similarly,
we can show that 
$e^{i(|\xi+\eta|^s + |\xi|^s + |\eta|^s)}\sigma_{\II}(\xi, \eta)$
and 
$e^{i(|\xi+\eta|^s + |\xi|^s + |\eta|^s)}\sigma_{\III}(\xi, \eta)$
 are also bilinear Fourier multipliers 
for $L^\infty \times L^\infty \to L^\infty$.

Next, we recall the duality argument.
By duality, the inequality \eqref{est_red} is equivalent to the inequality
\begin{align*}
\Big|
\int
T_{e^{i(|\xi|^s + |\eta|^s + |\xi + \eta|^s)}
\widetilde{\sigma}(\xi, \eta)}(f, g)(x)
h(x)
\,
dx
\Big|
\le
c \|\theta_1\|_{C^M} \|\theta_2\|_{C^M}
\|f\|_{L^\infty}
\|g\|_{L^\infty}
\|h\|_{L^1}.
\end{align*}
The left hand side can be written as
\begin{align*}
\Big|
c_0
\int
T_{e^{i(|\xi|^s + |\eta|^s + |\xi + \eta|^s)}
\theta_1(\xi) \theta_2(\eta)}(f, g)(x)
h(x)
\,
dx
\Big|
\end{align*}
or
\begin{align*}
\Big|
\sum_{j}
c_j
\int
T_{e^{i(|\xi|^s + |\eta|^s + |\xi + \eta|^s)}
\theta_1(2^{-j}\xi) \theta_2(2^{-j}\eta)}(f, g)(x)
h(x)
\,
dx
\Big|.
\end{align*}
Take $\theta_3 \in C^\infty_0(\R^n)$ such that 
\begin{align*}
\theta_3(-\zeta)
=
1
\;\;
\text{on}
\;\;
\supp \theta_1 + \supp \theta_2.
\end{align*}
Then it holds that for $j \in \N_0$
\begin{align*}
\int
T_{e^{i(|\xi|^s + |\eta|^s + |\xi + \eta|^s)}
\theta_1(2^{-j}\xi) \theta_2(2^{-j}\eta)}(f, g)(x)
h(x)
\,
dx
=
\int
S_{j} f(x)
\cdot
S_{j} g(x)
\cdot
S_{j} h(x)
\,
dx
\end{align*}
where, with rude descriptions, we wrote 
\begin{align*}
S_{j} f = e^{i |D|^s} \theta_1(2^{-j}D) f , 
\quad
S_{j} g = e^{i |D|^s} \theta_2(2^{-j}D) g , 
\quad
S_{j} h = e^{i |D|^s} \theta_3(2^{-j}D) h .
\end{align*}
Hence, the estimate \eqref{est_red} 
for $\widetilde{\sigma}$ of \eqref{sigma_case1}
and \eqref{sigma_case2}
is equivalent to the following two inequalities:
\begin{align*} 
\bigg|
\int
S_{0} f(x)
\cdot
S_{0} g(x)
\cdot
S_{0} h(x)
\,
dx
\bigg|
\le
c \|\theta_1\|_{C^M} \|\theta_2\|_{C^M}
\|f\|_{L^\infty}
\|g\|_{L^\infty}
\|h\|_{L^1}
\end{align*}
and
\begin{align*}
&
\sum_{j=1}^{\infty}
2^{jm} 
\bigg|
\int
S_{j} f(x)
\cdot
S_{j} g(x)
\cdot
S_{j} h(x)
\,
dx
\bigg|
\le
c \|\theta_1\|_{C^M} \|\theta_2\|_{C^M}
\|f\|_{L^\infty}
\|g\|_{L^\infty}
\|h\|_{L^1}.
\end{align*}
Furthermore, these inequality can be derived from
\begin{align} \label{L1norm_sneq1_easy}
\big\|
S_{0} f
\cdot
S_{0} g
\cdot
S_{0} h
\big\|_{L^1}
\le
c \|\theta_1\|_{C^M} \|\theta_2\|_{C^M}
\|f\|_{L^\infty}
\|g\|_{L^\infty}
\|h\|_{L^1}
\end{align}
and
\begin{align} \label{L1norm_sneq1}
&
\sum_{j=1}^{\infty}
2^{jm} 
\big\|
S_{j} f
\cdot
S_{j} g
\cdot
S_{j} h
\big\|_{L^1}
\le
c \|\theta_1\|_{C^M} \|\theta_2\|_{C^M}
\|f\|_{L^\infty}
\|g\|_{L^\infty}
\|h\|_{L^1},
\end{align}
respectively.

To prove \eqref{L1norm_sneq1_easy} and \eqref{L1norm_sneq1}, 
we will use the kernels of the operators $S_j$.
We define
\[
K_j = \big(e^{i|\xi|^s} \theta_{\ell}(2^{-j} \xi)\big)^{\vee}, 
\quad
\ell = 1, 2, 3,
\]
so that $S_j F = K_j * F$.
It follows from Lemma \ref{kernel_sneq1_norm} that 
\begin{align}
&
\|S_j F\|_{L^1}
\lesssim
\big\|
K_j
\big\|_{L^1}
\|F\|_{L^1}
\lesssim
\|\theta_{\ell}\|_{C^M}
2^{j \frac{ns}{2}}
\|F\|_{L^1},
\label{L1toL1}
\\
&
\big\|
S_j F
\big\|_{L^\infty}
\le
\big\| K_j \big\|_{L^1}
\| F \|_{L^\infty}
\lesssim
\|\theta_{\ell}\|_{C^M}
2^{j \frac{ns}{2}}
\| F \|_{L^\infty},
\label{LinftytoLinfty}
\end{align}
and
\begin{align} 
\big\|
S_j F
\big\|_{L^\infty}
\le
\big\| K_j \big\|_{L^\infty}
\| F \|_{L^1}
\lesssim
\|\theta_{\ell}\|_{C^M}
2^{j ( n - \frac{ns}{2} )}
\| F \|_{L^1}
\quad
\text{if}
\;\;
s < 1.
\label{LinftytoL1}
\end{align}

We will repeatedly use these estimates.

The proof of \eqref{L1norm_sneq1_easy} is simple.
In fact, 
\eqref{L1toL1} and
\eqref{LinftytoLinfty}
yield that 
\begin{align*}
&
\big\|
S_{0} f
\cdot
S_{0} g
\cdot
S_{0} h
\big\|_{L^1}
\le
\big\|
S_{0} f
\big\|_{L^\infty}
\big\|
S_{0} g
\big\|_{L^\infty}
\big\|
S_{0} h
\big\|_{L^1}
\lesssim
\|\theta_1\|_{C^M}
\|\theta_2\|_{C^M}
\|f\|_{L^\infty}
\|g\|_{L^\infty}
\|h\|_{L^1}.
\end{align*}
Thus, in what follows, we mainly consider the proof of \eqref{L1norm_sneq1}.
In Subsections \ref{cases<1} and \ref{cases>1},
we will give the proof of \eqref{L1norm_sneq1}
under the assumption $m < m_s$
for $s <1$ and $s > 1$, respectively.

We make some reduction.
Firstly, it is sufficient to show that the inequality \eqref{L1norm_sneq1} 
holds with $M$ being the largest integer of $M$'s appearing in the claims of Lemmas \ref{kernel_sneq1} and \ref{kernel_sneq1_norm}.
For this $M$, we may assume that 
$\|\theta_1\|_{C^{M}} = \|\theta_2\|_{C^{M}} = 1$
by homogeneity.
Secondly, by the atomic decomposition of $L^1$ (see Appendix)
and translation invariance,
we may assume that 
$h \in L^1$
satisfies
\begin{align*}
\| h \|_{L^{\infty}} \le r^{-n}, 
\quad
\supp h \subset \{ x \mid |x| \le r \}, 
\quad
0 < r \le 1.
\end{align*}
Consequently, in order to obtain \eqref{L1norm_sneq1},
it is sufficient to prove that the estimate
\begin{align} \label{goal_sneq1}
\begin{split}
&
\sum_{j=1}^{\infty}
2^{jm} 
\big\|
S_{j} f
\cdot
S_{j} g
\cdot
S_{j} h
\big\|_{L^1}
\le
c 
\|f\|_{L^\infty}
\|g\|_{L^\infty}
\end{split}
\end{align}
holds under the above assumptions.

\subsection{Case $s < 1$.} \label{cases<1}

In this subsection, we shall give a proof of \eqref{goal_sneq1}
for $s < 1$. 

\begin{prop} \label{s<1}
Let $0 < s < 1$.
Then 
\eqref{goal_sneq1} holds
for $m < m_s = - ns/2 - ns(1-s)$.
\end{prop}

\begin{proof}
We divide the $L^1$-norm of \eqref{goal_sneq1} into 
\begin{align*}
\begin{split}
&
\big\|
\cdots
\big\|_{L^1(\R^n)}
=
\big\|
\cdots
\big\|_{L^1(|x| \le 2r)}
+
\big\|
\cdots
\big\|_{L^1(2r \le |x| \le 4)}
+
\big\|
\cdots
\big\|_{L^1(|x| \ge 4)}.
\end{split}
\end{align*}
In the following argument, we will prove 
\begin{align}
&
\big\|
S_j f \cdot S_j g \cdot S_j h
\big\|_{L^1(|x| \le 2r)}
\lesssim
2^{-jm_s}
\min 
\big\{
(2^{j(1-s)}r)^{n(1-s)},
(2^{j(1-s)}r)^{-\frac{ns}{2}}
\big\}
\|f\|_{L^\infty}
\|g\|_{L^\infty},
\label{L1first}
\\
&
\big\|
S_j f \cdot S_j g \cdot S_j h
\big\|_{L^1(2r \le |x| \le 4)}
\lesssim
2^{-jm_s}
\|f\|_{L^\infty}
\|g\|_{L^\infty},
\label{L1second}
\\
&
\big\|
S_j f \cdot S_j g \cdot S_j h
\big\|_{L^1(|x| \ge 4)}
\lesssim
2^{j\frac{ns}{2}}
\|f\|_{L^\infty}
\|g\|_{L^\infty}.
\label{L1third}
\end{align}
Once these estimates are proved, then \eqref{goal_sneq1} follows for $m < m_s$.

{\it Proof of\/} \eqref{L1first}.
Before proving \eqref{L1first}, we prepare the following estimate:
\begin{align} \label{keyL2est_s<1}
\big\|
S_j f
\big\|_{L^2(|x| \le A)}
\lesssim
A^{\frac{n(1-s)}{2}}
\|f\|_{L^\infty}
\quad
\text{if}\;\;
0 < A \le 8.
\end{align}
Here the implicit constant is depending only on $n$ and $s$.

To show \eqref{keyL2est_s<1}, we split $f$ into
\begin{align*} 
f = f \ichi_{\{ |y| \le CA^{1-s} \}} + f \ichi_{\{ |y| > CA^{1-s} \}}
= f_{A}^{0} + f_{A}^{1}
\end{align*}
with $C = 2 \cdot 8^s$.
First, Plancherel's theorem gives
\begin{align*}
\big\| S_{j} f_{A}^{0} \big\|_{L^2( |x| \le A )} 
\le
\big\| S_{j} f_{A}^{0} \big\|_{L^2(\R^n)} 
\le
\| f_{A}^{0} \|_{L^2(\R^n)} 
\lesssim
A^{\frac{n(1-s)}{2}}
\|f\|_{L^\infty}.
\end{align*}
Next, since
$|y| \approx |x-y|$
if $|x| \le A$ and $|y| > C A^{1-s} \ge C 8^{-s} A = 2A$,
we have
by Lemma \ref{kernel_sneq1} (1)
\begin{align*}
&
\big\| S_{j} f_{A}^{1} \big\|_{L^2( |x| \le A )} 
\le
\Big\| \int_{|y| > C A^{1-s}} 
\Big| \Big( e^{ i |\xi|^{s} } \theta (2^{-j} \xi ) \Big)^{\vee} (x-y) \Big| 
\, 
|f(y)| \,dy \Big\|_{L^2( |x| \le A )} 
\\
&
=
\Big\| 
\Big(
\int_{\substack{ |y| > C A^{1-s} \\ |x-y| \le 1 }} 
+
\int_{\substack{ |y| > C A^{1-s} \\ |x-y| \ge 1 }} 
\Big)
\Big| \Big( e^{ i |\xi|^{s} } \theta (2^{-j} \xi ) \Big)^{\vee} (x-y) \Big| 
\, 
|f(y)| \,dy \Big\|_{L^2( |x| \le A )} 
\\
&
\lesssim
\Big(
\Big\| 
\int_{|y| > C A^{1-s}} 
|y|^{-\frac{n}{2} - \frac{n}{2(1-s)}} 
\,
dy 
\Big\|_{L^2( |x| \le A )} 
+
\Big\| 
\int_{|x-y| > 1} 
|x-y|^{-(n+s)} 
\,
dy 
\Big\|_{L^2( |x| \le A )} 
\Big)
\|f\|_{L^\infty}
\\
&
\lesssim
\Big(
A^{\frac{n(1-s)}{2}}
+
A^{\frac{n}{2}}
\Big)
\|f\|_{L^\infty}
\approx
A^{\frac{n(1-s)}{2}}
\|f\|_{L^\infty}.
\end{align*}
Thus \eqref{keyL2est_s<1} is proved.

Now, we prove \eqref{L1first}.
By Plancherel's theorem, we have
\begin{align*}
\big\|
S_j h
\big\|_{L^2}
\le
\|h\|_{L^2}
\le
r^{-\frac{n}{2}}.
\end{align*}
Combining this,
\eqref{LinftytoLinfty},
and
\eqref{keyL2est_s<1} with $A = 2r$,
we have 
\begin{align*}
&
\big\|
S_j f \cdot S_j g \cdot S_j h
\big\|_{L^1(|x| \le 2r)}
\le
\big\|
S_j f
\big\|_{L^2(|x| \le 2r)}
\big\|
S_j g
\big\|_{L^\infty}
\big\|
S_j h
\big\|_{L^2}
\\
&
\lesssim
r^{\frac{n(1-s)}{2}}
\|f\|_{L^\infty}
\cdot
2^{j\frac{ns}{2}}
\|g\|_{L^\infty}
\cdot
r^{-\frac{n}{2}}
=
2^{-j(m_s + \frac{ns(1-s)}{2})}
(2^{j(1-s)}r)^{-\frac{ns}{2}}
\|f\|_{L^\infty}
\|g\|_{L^\infty}
\\
&
\le
2^{-jm_s}
(2^{j(1-s)}r)^{-\frac{ns}{2}}
\|f\|_{L^\infty}
\|g\|_{L^\infty}.
\end{align*}
On the other hand, 
by combining 
\eqref{LinftytoL1}
and
\eqref{keyL2est_s<1} with $A = 2r$,
it holds that 
\begin{align*}
&
\big\|
S_j f \cdot S_j g \cdot S_j h
\big\|_{L^1(|x| \le 2r)}
\le
\big\|
S_j f
\big\|_{L^2(|x| \le 2r)}
\big\|
S_j g
\big\|_{L^2(|x| \le 2r)}
\big\|
S_j h
\big\|_{L^\infty}
\\
&
\lesssim
r^{\frac{n(1-s)}{2}}
\|f\|_{L^\infty}
\cdot
r^{\frac{n(1-s)}{2}}
\|g\|_{L^\infty}
\cdot
2^{j(n - \frac{ns}{2})}
\|h\|_{L^1}
\lesssim
2^{-jm_s} 
(2^{j(1-s)}r)^{n(1-s)}
\|f\|_{L^\infty}
\|g\|_{L^\infty},
\end{align*}
which completes the proof of \eqref{L1first}.

{\it Proof of\/} \eqref{L1second}.
To prove \eqref{L1second}, we divide the $L^1$ norm as
\begin{align*}
\big\|
S_j f \cdot S_j g \cdot S_j h
\big\|_{L^1(2r \le |x| \le 4)}
\le
\sum_{ k \in \N : 2^k r \le 4}
\big\|
S_j f \cdot S_j g \cdot S_j h
\big\|_{L^1(\widetilde{E}_k)},
\end{align*} 
where 
\begin{align*}
\widetilde{E}_k = B(0, 2^{k+1} r) \setminus B(0, 2^k r)
= \{x \in \R^n \mid 2^{k}r \le |x| < 2^{k+1}r \}.
\end{align*}
Then, we shall prove  the following estimate: 
\begin{align} \label{Linftyest3_s<1}
\big\|
S_j h
\big\|_{L^\infty(\widetilde{E}_k)}
\lesssim
(2^{k} r)^{-\frac{n}{2} - \frac{n}{2(1-s)}}
\quad
\text{for $k \in \N$ satisfying $2^k r \le 4$},
\end{align}
where the implicit constant is independent of $j$, $k$ and $r$.

To prove \eqref{Linftyest3_s<1},
observe that $|x-y| \approx |x|$ 
if $x \in \widetilde{E}_k$, $k \in \N$, and $|y| \le r$.
Then by Lemma \ref{kernel_sneq1} (1)
\begin{align*}
&
\big\|
S_j h
\big\|_{L^\infty(\widetilde{E}_k)}
\le
\Big\|
\Big(
\int_{\substack{ |y| \le r \\ |x-y| \le 1}}
+
\int_{\substack{ |y| \le r \\ |x-y| \ge 1}}
\Big)
\Big|
\Big(
e^{i|\xi|^s}
\theta(2^{-j} \xi)
\Big)^{\vee}
(x-y)
\Big|
|h(y)|
\,
dy
\Big\|_{L^\infty(\widetilde{E}_k)}
\\
&
\lesssim
\Big\|
\int_{\substack{|y| \le r \\ |x-y| \le 1}}
|x|^{-\frac{n}{2} - \frac{n}{2(1-s)}}
|h(y)|
\,
dy
\Big\|_{L^\infty(\widetilde{E}_k)}
+
\Big\|
\int_{\substack{|y| \le r \\ |x-y| > 1}}
|x-y|^{-(n+s)}
|h(y)|
\,
dy
\Big\|_{L^\infty(\widetilde{E}_k)}
\\
&
\lesssim
(2^{k} r)^{-\frac{n}{2} - \frac{n}{2(1-s)}}
+
1
\approx
(2^{k} r)^{-\frac{n}{2} - \frac{n}{2(1-s)}}.
\end{align*}
Thus \eqref{Linftyest3_s<1} is proved.

Now, it follows from 
\eqref{LinftytoL1}
and
\eqref{keyL2est_s<1} with $A = 2^{k+1}r$
that
\begin{align*}
&
\big\|
S_j f \cdot S_j g \cdot S_j h
\big\|_{L^1(\widetilde{E}_k)}
\le
\big\|
S_j f
\big\|_{L^2(|x| \le 2^{k+1}r)}
\big\|
S_j g
\big\|_{L^2(|x| \le 2^{k+1}r)}
\big\|
S_j h
\big\|_{L^\infty}
\\
&
\lesssim
(2^k r)^{\frac{n(1-s)}{2}}
\|f\|_{L^\infty}
\cdot
(2^k r)^{\frac{n(1-s)}{2}}
\|g\|_{L^\infty}
\cdot
(2^j)^{n - \frac{ns}{2}}
=
2^{-jm_s}
(2^{j(1-s)} 2^k r)^{n(1-s)}
\|f\|_{L^\infty}
\|g\|_{L^\infty}.
\end{align*}
On the other hand, 
by 
\eqref{Linftyest3_s<1}
and
\eqref{keyL2est_s<1} with $A = 2^{k+1}r$,
\begin{align*}
&
\big\|
S_j f \cdot S_j g \cdot S_j h
\big\|_{L^1(\widetilde{E}_k)}
\le
\big\|
S_j f
\big\|_{L^2(|x| \le 2^{k+1}r)}
\big\|
S_j g
\big\|_{L^2(|x| \le 2^{k+1}r)}
\big\|
S_j h
\big\|_{L^\infty(\widetilde{E}_k)}
\\
&
\lesssim
(2^k r)^{n(1-s)-\frac{n}{2} - \frac{n}{2(1-s)}}
\|f\|_{L^\infty}
\|g\|_{L^\infty}
=
2^{-jm_s}
(2^{j(1-s)} 2^k r)^{n(1-s)-\frac{n}{2} - \frac{n}{2(1-s)}}
\|f\|_{L^\infty}
\|g\|_{L^\infty}.
\end{align*}
We remark that $n(1-s)-\frac{n}{2} - \frac{n}{2(1-s)} < 0$ if $0 < s < 1$.
Hence, these two estimates give
\begin{align*}
&
\sum_{k \in \N : 2^k r \le 4}
\big\|
S_j f \cdot S_j g \cdot S_j h
\big\|_{L^1(\widetilde{E}_k)}
\\
&
\lesssim
2^{-jm_s}
\Bigg(
\sum_{\substack{k \in \N : 2^k r \le 4 \\ 2^{j(1-s)}2^k r \le 1}}
(2^{j(1-s)} 2^k r)^{n(1-s)}
+
\sum_{\substack{k \in \N : 2^k r \le 4 \\ 2^{j(1-s)}2^k r > 1}}
(2^{j(1-s)} 2^k r)^{n(1-s)-\frac{n}{2} - \frac{n}{2(1-s)}}
\Bigg)
\|f\|_{L^\infty}
\|g\|_{L^\infty}
\\
&
\lesssim
2^{-jm_s}
\|f\|_{L^\infty}
\|g\|_{L^\infty}.
\end{align*}
This completes the proof of  \eqref{L1second}.

{\it Proof of\/} \eqref{L1third}.
To prove \eqref{L1third}, we split the $L^1$ as follows:
\begin{align*}
\big\|
S_j f \cdot S_j g \cdot S_j h
\big\|_{L^1(|x| \ge 4)}
=
\sum_{ k =2 }^{\infty}
\big\|
S_j f \cdot S_j g \cdot S_j h
\big\|_{L^1(E_k)},
\end{align*} 
where 
\begin{align} \label{annulus}
E_k = B(0, 2^{k+1}) \setminus B(0, 2^k)
= \{x \in \R^n \mid 2^{k} \le |x| < 2^{k+1} \}.
\end{align}
Then, it is enough to show the following two estimates:
\begin{align}
&
\big\|
S_j f \cdot S_j g
\big\|_{L^1(E_k)}
\lesssim
2^{kn}
2^{j\frac{ns}{2}}
\|f\|_{L^\infty}
\|g\|_{L^\infty},
\label{bilinearest1_s<1}
\\
&
\big\|
S_j h
\big\|_{L^{\infty}(E_k)}
\lesssim
2^{-k(n+s)}.
\label{Linftyest1_s<1}
\end{align}
In fact, these estimates yield 
\begin{align*}
&
\big\|
S_j f \cdot S_j g \cdot S_j h
\big\|_{L^1(|x| \ge 4)}
=
\sum_{k =2}^\infty
\big\|
S_j f \cdot S_j g \cdot S_j h
\big\|_{L^1(E_k)}
\\
&
\lesssim
\sum_{k =2}^\infty
\big\|
S_j f \cdot S_j g
\big\|_{L^1(E_k)}
\big\|
S_j h
\big\|_{L^\infty(E_k)}
\lesssim
2^{j\frac{ns}{2}}
\Big(
\sum_{k =2}^\infty
2^{-ks}
\Big)
\|f\|_{L^\infty}
\|g\|_{L^\infty}
\lesssim
2^{j\frac{ns}{2}}
\|f\|_{L^\infty}
\|g\|_{L^\infty},
\end{align*}
and thus \eqref{L1third} will follow.

We now prove \eqref{bilinearest1_s<1}.
For each $k \ge 2$, 
we decompose $f$ as
\[
f = f \ichi_{\{|x| \le 2^{k+2}\}} + f \ichi_{\{|x| > 2^{k+2}\}}
=
f_k^{0} + f_k^{1}.
\]
It follows from Plancherel's theorem that 
\begin{align}\label{L2est1_s<1}
\big\|
S_j f^0_k
\big\|_{L^2(E_k)}
\le
\big\|
S_j f^0_k
\big\|_{L^2(\R^n)}
\le
\|f^0_k\|_{L^2}
\lesssim
2^{k\frac{n}{2}}
\|f\|_{L^\infty}.
\end{align}
On the other hand,
from the estimates $|x-y| \ge 1$ and
$|x-y| \approx |y|$ for $x \in E_k$ and $|y| \ge 2^{k+2} \ge 1$,
and from Lemma \ref{kernel_sneq1} (1), 
it follows that
\begin{align} \label{Linftyest2_s<1}
\begin{split}
\big\|
S_j f^1_k
\big\|_{L^\infty(E_k)}
&
\le
\Big\|
\int_{|y| \ge 2^{k+2}}
\Big|
\Big(
e^{i|\xi|^s}
\theta(2^{-j} \xi)
\Big)^{\vee}
(x-y)
\Big|
|f(y)|
\,
dy
\Big\|_{L^\infty(E_k)}
\\
&
\lesssim
\Big\|
\int_{|y| \ge 2^{k+2}}
|y|^{-(n+s)}
\,
dy
\Big\|_{L^\infty(E_k)}
\|f\|_{L^\infty}
\lesssim
\|f\|_{L^\infty}.
\end{split}
\end{align}
Hence, \eqref{L2est1_s<1} gives
\begin{align*}
\big\|
S_j f_k^0 \cdot S_j g
\big\|_{L^1(E_k)}
&
\lesssim
2^{k\frac{n}{2}}
\big\|
S_j f_k^0
\big\|_{L^2} 
\big\|
S_j g 
\big\|_{L^\infty}
\lesssim
2^{kn}
2^{j\frac{ns}{2}}
\|f\|_{L^\infty}
\|g\|_{L^\infty},
\end{align*}
and \eqref{Linftyest2_s<1} gives
\begin{align*}
\big\|
S_j f_k^1 \cdot S_j g
\big\|_{L^1(E_k)}
&
\le
2^{kn}
\big\|
S_j f_k^1
\big\|_{L^\infty(E_k)}
\,
\big\|
S_j g 
\big\|_{L^\infty}
\lesssim
2^{kn}
2^{j\frac{ns}{2}}
\|f\|_{L^\infty}
\|g\|_{L^\infty}.
\end{align*}
Combining these estimates, we obtain \eqref{bilinearest1_s<1}.

We next give a proof of \eqref{Linftyest1_s<1}.
Since $|x-y| \ge 1$ and $|x-y| \approx |x|$ for $x \in E_k$ and $|y| \le r \le 1$,
we have by Lemma \ref{kernel_sneq1} (1)
\begin{align*}
\big\|
S_j h
\big\|_{L^{\infty}(E_k)}
&
\le
\Big\|
\int_{|y| \le r}
\Big|
\Big(
e^{i|\xi|^s}
\theta(2^{-j} \xi)
\Big)^{\vee}
(x-y)
\Big|
|h(y)|
\,
dy
\Big\|_{L^{\infty}(E_k)}
\\
&
\lesssim
\Big\|
\int_{|y| \le r}
|x|^{-(n+s)}
|h(y)|
\,
dy
\Big\|_{L^{\infty}(E_k)}
\lesssim
2^{-k(n+s)}.
\end{align*}
Thus \eqref{Linftyest1_s<1} is proved.
We complete the proof of Proposition \ref{s<1}.
\end{proof}

\subsection{Case $s > 1$} \label{cases>1}

In this subsection, we shall give a proof of \eqref{goal_sneq1}
for $s > 1$.

\begin{prop} \label{thm_s>1}
Let $1 < s < \infty$.
Then 
\eqref{goal_sneq1} holds for 
\[
m < m_{s} =
\begin{cases}
-{ns}/{2}, & \textrm{if} \;\; 1 < s \le 2, \\
-n(s-1), & \textrm{if} \;\; 2 < s < \infty.
\end{cases}
\]
\end{prop}

\begin{proof}
Since $m<m_{s}$ is assumed,
in order to obtain \eqref{goal_sneq1},
it is sufficient to prove that
\begin{align} \label{goal_s>1}
\big\| S_j f \cdot S_j g \cdot S_j h \big\|_{L^1(\R^n)}
\lesssim
2^{- j m_s}
\|f\|_{L^\infty}
\|g\|_{L^\infty}
\end{align}
holds for sufficiently large $j \in \N_{0}$.
In fact, for $j \lesssim 1$, \eqref{goal_s>1} immediately follows from H\"older's inequality,
\eqref{L1toL1}, and \eqref{LinftytoLinfty}.

To obtain this,
we first decompose 
the left hand side as
\begin{align*}
\big\| \cdots \big\|_{L^1(\R^n)}
\le 
\big\| \cdots \big\|_{L^1( A_{j} )}
+\sum_{k \in \N : 2^{k} \ge C2^{j(s-1)+1}}
\big\| \cdots \big\|_{L^1( E_k )} ,
\end{align*}
where the set $A_j$ is defined by
\begin{align*}
A_{j} = \{ x \in \R^n \mid |x| \le C 2^{j(s-1)+2} \}
\end{align*}
with
the constant $C$ appearing in Lemma \ref{kernel_sneq1} (2),
and the set $E_k$ is defined by \eqref{annulus}.
Then, 
the following inequalities
immediately complete 
the proof of \eqref{goal_s>1}:
\begin{align} \label{Aj_s>1}
&
\big\| S_j f \cdot S_j g \cdot S_j h \big\|_{L^1( A_{j} )}
\lesssim 
2^{- j m_{s}}
\|f\|_{L^\infty}
\|g\|_{L^\infty},
\\
&
\big\| S_j f \cdot S_j g \cdot S_j h \big\|_{L^1(E_{k})} 
\lesssim 
2^{-ks}
\|f\|_{L^\infty}
\|g\|_{L^\infty}
\quad
\text{if $k \in \N$ satisfies $2^{k} \ge C2^{j(s-1)+1}$} . 
\label{Bk_s>1}
\end{align}

{\it Proof of\/} \eqref{Aj_s>1}.
We first display the two inequalities
which conclude \eqref{Aj_s>1}:
\begin{align} \label{Aj_f_s>1}
\| S_{j} f \|_{L^2( |x| \le A )} 
\lesssim 
A^{\frac{n}{2}} 
\|f\|_{L^{\infty}} 
\quad
\text{if $A$ satisfies $A \ge C 2^{j(s-1)+ 1}$} 
\end{align}
and
\begin{align} \label{Aj_h_s>1}
\| S_{j} h \|_{L^{\infty}(A_{j})} 
\lesssim 
\max\{ 1, 2^{j(n-\frac{ns}{2})}\} .
\end{align}
If these are proved,
we have by
H\"older's inequality,
\eqref{Aj_f_s>1} with $A=C 2^{j(s-1)+ 2}$,
and \eqref{Aj_h_s>1}
\begin{align*}
&
\big\| S_j f \cdot S_j g \cdot S_j h \big\|_{L^1( A_{j} )}
\le
\| S_{j} f \|_{L^2( A_{j} )} \, 
\| S_{j} g \|_{L^2( A_{j} )}\, 
\| S_{j} h \|_{ L^{\infty} (\R^n) }
\\
& 
\lesssim 
2^{j \frac{n(s-1)}{2}} 
\|f\|_{L^{\infty}}
\cdot
2^{j \frac{n(s-1)}{2}} 
\|g\|_{L^{\infty}} 
\cdot
\max\{ 1, 2^{j(n-\frac{ns}{2})}\} 
=
2^{- j m_s}  
\|f\|_{L^{\infty}} \|g\|_{L^{\infty}},
\end{align*}
which implies \eqref{Aj_s>1}.

Now, we first show \eqref{Aj_f_s>1}.
Decompose $f$ as
\begin{align*} 
f = f \ichi_{\{ |x| \le 2A \}} + f \ichi_{\{ |x| > 2A \}}
= f_{A}^{0} + f_{A}^{1} .
\end{align*}
For the factor $f_{A}^{0}$,
Plancherel's theorem gives
\begin{align*}
\| S_{j} f_{A}^{0} \|_{L^2( |x| \le A )} 
\le
\| S_{j} f_{A}^{0} \|_{L^2(\R^n)} 
\le
\| f_{A}^{0} \|_{L^2(\R^n)} 
\lesssim
A^{\frac{n}{2}} 
\|f\|_{L^{\infty}} .
\end{align*}
For the factor $f_{A}^{1}$,
since $|y| \approx |x-y| \ge A \ge C2^{j(s-1)}$
for $|y| > 2A$ and $|x| \le A$,
we have
by Lemma \ref{kernel_sneq1} (2)
\begin{align*}
\| S_{j} f_{A}^{1} \|_{L^2( |x| \le A )} 
&\le
\Big\| \int_{|y| > 2A} 
\Big| \Big( e^{ i |\xi|^{s} } \theta (2^{-j} \xi ) \Big)^{\vee} (x-y) \Big| 
\, 
|f(y)| \,dy \Big\|_{L^2( |x| \le A )} 
\\
&
\lesssim
\Big\| \int_{|y| > 2A} |y|^{-(n+s)} \,dy \Big\|_{L^2( |x| \le A )} 
\|f\|_{L^{\infty}}
\lesssim
A^{\frac{n}{2}}
\|f\|_{L^{\infty}}.
\end{align*}
Combining these two inequalities,
we obtain \eqref{Aj_f_s>1}.

We next consider \eqref{Aj_h_s>1}.
Observe that 
$- \frac{n}{2} + \frac{n}{2(s-1)}$
is nonnegative
if $1 < s \le 2$
and is negative if $s > 2$.
Since $1 +|x-y| \lesssim 2^{j(s-1)}$ if $|x| \le C2^{j(s-1)+2}$ and $|y| \le r \le 1$,
we have by
Lemma  \ref{kernel_sneq1} (2)
\begin{align*}
\| S_{j} h \|_{L^{\infty}(A_{j})}
&\le 
\Big\| \int_{|y| \le r} 
\Big| \Big( e^{ i |\xi|^{s} } \theta (2^{-j} \xi ) \Big)^{\vee} (x-y) \Big| 
\, 
|h(y)| \,dy 
\Big\|_{L^\infty( A_{j} )} 
\\&\lesssim 
\Big\| \int_{|y| \le r} 
\big( 1+|x-y| \big)^{ - \frac{n}{2} + \frac{n}{2(s-1)} }
\, 
|h(y)| \,dy \Big\|_{L^\infty( A_{j} )} 
\lesssim
\max\{ 1, 2^{j(n-\frac{ns}{2})}\},
\end{align*}
which yields \eqref{Aj_h_s>1}.

{\it Proof of\/} \eqref{Bk_s>1}.
We prepare the inequality
\begin{align} \label{Bk_h_s>1}
\| S_{j} h \|_{L^{\infty} (E_k)}
\lesssim
2^{-k(n+s)} 
\quad
\text{if $k \in \N$ satisfies that $2^{k} \ge C2^{j(s-1)+1}$} .
\end{align}
Once this is proved,
we have 
by H\"older's inequality,
\eqref{Aj_f_s>1} with $A=2^{k+1}$,
and \eqref{Bk_h_s>1}
\begin{align*}
\big\| S_j f \cdot S_j g \cdot S_j h \big\|_{L^1(E_k)} 
&
\le
\| S_{j} f \|_{L^2(E_k)}
\| S_{j} g \|_{L^2(E_k)}
\| S_{j} h \|_{L^{\infty} (E_k)} 
\\
&
\lesssim
2^{k \frac{n}{2}} 
\|f\|_{L^{\infty}}
\cdot
2^{k \frac{n}{2}} 
\|g\|_{L^{\infty}}
\cdot
2^{-k(n+s)} 
=
2^{-ks} 
\|f\|_{L^{\infty}}
\|g\|_{L^{\infty}}
\end{align*}
if $k$ satisfies that $2^{k} \ge C2^{j(s-1)+1}$,
which implies \eqref{Bk_s>1}.

Now, we shall show \eqref{Bk_h_s>1}.
Observe that
$|x| \approx |x-y| > 2^{k-1} \ge C 2^{j(s-1)}$
if 
$x \in E_k$, 
$|y| \le 1$,
and $2^{k} \ge C2^{j(s-1)+1}$.
Then, since
$\supp h \subset \{ |x| \le r \}$ with $0 < r \le 1$,
we have by Lemma \ref{kernel_sneq1} (2)
\begin{align*}
\| S_{j} h \|_{L^{\infty} (E_k)}
&=
\Big\| \int_{|y| \le r} 
\Big| \Big( e^{ i |\xi|^{s} } \theta (2^{-j} \xi ) \Big)^{\vee} (x-y) \Big| 
\, 
|h(y)| \,dy 
\Big\|_{L^{\infty}( E_k )} 
\\&\lesssim
\Big\| \int_{|y| \le r} |x|^{-(n+s)} |h(y)| \,dy \Big\|_{L^{\infty} ( E_k )} 
\lesssim
2^{-k(n+s)}, 
\end{align*}
which finishes the proof of \eqref{Bk_h_s>1}.
Thus we complete the proof of Proposition \ref{thm_s>1}.
\end{proof}

\section{Sharpness of the order $m_s$} \label{sharpness}

We take two functions $ \theta, \phi \in C^\infty_0(\R^n)$ 
such that
\begin{align*}
&
\theta(\xi) = 1 
\;\;\text{on}\;\;
\{1/8 \le |\xi| \le 8\} ,
\quad
\supp \theta \subset \{1/10 \le |\xi| \le 10\} ,
\\&
\phi(\xi) = 1 
\;\;\text{on}\;\;
\{ |\xi| \le 2\},
\quad
\supp \phi \subset \{ |\xi| \le 3\} .
\end{align*}
Let $m \in \R$.
For $j \in \N$, we define $\sigma_j$ by
\[
\sigma_j(\xi, \eta)
=
2^{jm}
\theta(2^{-j} \xi)
\phi(2^{-j}\eta).
\]
Since $1 + |\xi| + |\eta| \approx 2^{j}$ if 
$(\xi, \eta) \in \supp \theta(2^{-j} \cdot) \times \supp \phi(2^{-j} \cdot)$, 
it holds that 
\begin{align*}
\big|
\partial^{\alpha}_{\xi}
\partial^{\beta}_{\eta}
\sigma_j(\xi, \eta)
\big|
\le
C_{\alpha, \beta}
(1+|\xi|+|\eta|)^{m-|\alpha|-|\beta|}
\end{align*}
uniformly in $j \in \N$,
and hence we have $\sigma_j \in S^m_{1, 0}(\R^{2n})$.
Therefore, if the claim
\begin{align} \label{goal_LinfLinfbmo}
e^{i(|\xi|^s + |\eta|^s + |\xi + \eta|^s)}
\sigma(\xi, \eta)
\in
\calM(L^\infty \times L^\infty \to bmo)
\;\;
\text{for all}
\;\;
\sigma \in S^m_{1, 0}(\R^{2n})
\end{align}
holds, then we see that the inequality
\begin{align} \label{dualform}
2^{jm} 
\bigg|
\int
e^{i|D|^s}
\theta(2^{-j}D) f(x)
\cdot
e^{i|D|^s}
\phi(2^{-j}D) g(x)
\cdot
e^{i|D|^s}
h(x)
\,
dx
\bigg|
\le
c 
\|f\|_{L^\infty}
\|g\|_{L^\infty}
\|h\|_{h^1}
\end{align}
holds uniformly in $j \in \N$.

On the other hand,
if we define  $\sigma^{*1}_j$ by 
\[
\sigma^{*1}_j(\xi, \eta) = \sigma_j(-\xi-\eta, \eta) 
= 2^{jm} \theta(2^{-j}(-\xi-\eta)) \phi(2^{-j} \eta),
\] 
then it holds that $\sigma^{*1}_j \in S^m_{1, 0}(\R^{2n})$ uniformly in $j$. 
Hence, 
by taking $\sigma = \sigma^{*1}_j$ in \eqref{goal_LinfLinfbmo},
we see that the claim \eqref{goal_LinfLinfbmo} also implies that the inequality
\begin{align} \label{dualform*}
2^{jm} 
\bigg|
\int
e^{i|D|^s} f(x)
\cdot
e^{i|D|^s}
\phi(2^{-j}D) g(x)
\cdot
e^{i|D|^s}
\theta(2^{-j}D)
h(x)
\,
dx
\bigg|
\le
c 
\|f\|_{L^\infty}
\|g\|_{L^\infty}
\|h\|_{h^1}
\end{align}
holds uniformly in $j \in \N$.
\subsection{Case $s < 1$}

In this subsection, we assume that $s < 1$.
We shall prove that 
the condition $m \le m_s$ is necessary 
to hold the inequality
\begin{align} \label{strongerineq}
2^{jm}
\big\|
e^{i|D|^s} f
\cdot
e^{i|D|^s}
\phi(2^{-j}D) g
\cdot
e^{i|D|^s} 
\theta(2^{-j}D) h
\big\|_{L^1}
\le
c
\|f\|_{L^\infty}
\|g\|_{L^\infty}
\|h\|_{h^1},
\end{align}
which is stronger than \eqref{dualform*}.

\begin{prop} \label{prop_s<1}
Let $0 < s < 1$ and $m \in \R$.
If \eqref{strongerineq} holds, then $m \le m_s$.
\end{prop}

\begin{proof}
Let $\psi \in C^{\infty}_0(\R^n)$ be such that 
\begin{align*}
\supp \psi \subset \{1/2 \le |\xi| \le 2\},
\quad
\psi \neq 0
\;\;\text{on}\;\;
\{2/3 \le |\xi| \le 3/2\}.
\end{align*}
We set
\begin{align*}
F_j(x)
=
\Big(e^{-i|\xi|^s} \psi(2^{-j(1-s)} \xi)\Big)^{\vee}(x)
\quad
\text{and}
\quad
G_j(x)
=
\Big(\psi(2^{-j} \xi)\Big)^{\vee}(x).
\end{align*}
Obviously, $\|G_j\|_{h^1} \lesssim 1$.
Moreover, 
since $s < 1$, we have by Proposition \ref{prop_psi_KMST1}
\begin{align*}
\big\| F_j \big\|_{L^\infty}
\lesssim
2^{j(1-s)(n - \frac{ns}{2})}.
\end{align*}
On the other hand, it follows from the choice of the functions $\theta$ and
$\phi$ that  
\begin{align*}
&
e^{i|D|^s} \phi(2^{-j} D)F_j(x)
=
e^{i|D|^s} F_j(x)
=
\Big(
\psi(2^{-j(1-s)} \xi)
\Big)^{\vee}(x)
=
2^{j(1-s)n} (\psi)^{\vee} (2^{j(1-s)}x),
\\
&
e^{i|D|^s} \theta(2^{-j} D)G_j(x)
=
\Big(
e^{i|\xi|^s}\psi(2^{-j} \xi)
\Big)^{\vee}(x).
\end{align*}
Hence, if we test \eqref{strongerineq} to the functions 
$(f, g, h) = (F_j, F_j, G_j)$,
then we have
\begin{align*}
2^{j m}
2^{2j(1-s)n}
\Big\|
\big\{
(\psi)^{\vee}(2^{-j(1-s)} \cdot)
\big\}^{2}
\cdot
\Big(
e^{i|\xi|^s}\psi(2^{-j} \xi)
\Big)^{\vee}
\Big\|_{L^1}
\lesssim
2^{2 j (1-s)(n - \frac{ns}{2})}.
\end{align*}
Since we now assume that $\psi \neq 0$ on $\{2/3 \le |\xi| \le 3/2\}$,
we have by \eqref{Kjxradial} of Proposition \ref{prop_psi_KMST1} 
\begin{align*}
&
\Big\|
\big\{
(\psi)^{\vee}(2^{j(1-s)} \cdot)
\big\}^{2}
\cdot
\Big(
e^{i|\xi|^s}\psi(2^{-j} \xi)
\Big)^{\vee}
\Big\|_{L^1}
\\
&
\ge
\int_{s(\frac{3}{2})^{s-1} \le 2^{j(1-s)}|x| \le s(\frac{3}{2})^{1-s}}
\big|
(\psi)^{\vee}(2^{j(1-s)} x)
\big|^2
\,
\Big|
\Big(
e^{i|\xi|^s}\psi(2^{-j} \xi)
\Big)^{\vee} (x)
\Big|
\,
dx
\\
&
\approx
2^{j(n - \frac{ns}{2})}
\int_{s(\frac{3}{2})^{s-1} \le 2^{j(1-s)}|x| \le s(\frac{3}{2})^{1-s}}
\big|
(\psi)^{\vee}(2^{j(1-s)} x)
\big|^2
\,
dx
\\
&
\approx
2^{j(n - \frac{ns}{2})}
2^{-jn(1-s)}.
\end{align*}
for sufficiently large $j$.
Hence we obtain
\begin{align*}
2^{j m}
2^{2j(1-s)n}
2^{j(n - \frac{ns}{2})}
2^{-jn(1-s)}
\lesssim
2^{2 j (1-s)(n - \frac{ns}{2})},
\end{align*}
which concludes that $m \le -ns/2 -ns(1-s) = m_s$.
\end{proof}

\begin{rem}
It should be emphasized that 
although \eqref{goal_LinfLinfbmo} implies \eqref{dualform*}
but the implication \eqref{dualform*} $\Rightarrow$ \eqref{strongerineq}
is highly nontrivial (or may not be true). 
Hence Proposition \ref{prop_s<1} is not enough to 
deduce the claim that \eqref{goal_LinfLinfbmo} holds only when $m \le m_s$. 

\end{rem}

\subsection{Case $s > 1$}

In this subsection, we shall prove that 
the claim \eqref{goal_LinfLinfbmo} holds only if $m \le m_s$ 
when $s > 1$. 

\begin{prop} \label{prop_s>1}
Let $1 < s < \infty$ and $m \in \R$.
If \eqref{goal_LinfLinfbmo} holds, 
then $m \le m_s$.
\end{prop}

\begin{proof}
As shown in the beginning of this section,
 \eqref{goal_LinfLinfbmo} implies
\eqref{dualform}.
Hence, 
it is sufficient to prove that \eqref{dualform} 
holds only if $m \le m_s$.

{\it Proof of the case $1 < s \le 2$.}
Take real-valued radial functions
$\psi_1, \psi_2, \psi_3 \in C^\infty_0(\R^n)$ 
such that 
\begin{align*}
&
\psi_1 \ge 0, 
\quad
\psi_1 = 1 \;\;\text{on}\;\;
\{ 1/4 \le |\xi| \le 4\}, 
\quad
\supp \psi_1 \subset \{ 1/8 \le |\xi| \le 8 \},
\\&
\psi_2 \ge 0, 
\quad
\psi_2 = 1 \;\;\text{on}\;\;
\{ 2/3 \le |\xi| \le 3/2\}, 
\quad
\supp \psi_2 \subset \{ 1/2 \le |\xi| \le 2 \},
\\&
\psi_3 \ge 0, 
\quad
\psi_3 = 1 \;\;\text{on}\;\;
\{ 1/2 \le |\xi| \le 2\}, 
\quad
\supp \psi_3 \subset \{ 1/4 \le |\xi| \le 4 \},
\end{align*}
and set
\begin{align*}
F_{j}^{\ell} (x)
=
\Big(
e^{-i|\xi|^s}
\psi_{\ell} (2^{-j}\xi)
\Big)^{\vee} (x)
\end{align*}
for $j \in \N$ and $\ell = 1,2,3$.
Then, by Proposition \ref{prop_psi_KMST1},
we have 
\begin{align*}
\| F_{j}^{\ell} \|_{L^\infty}
\lesssim
2^{j(n-\frac{ns}{2})},
\;\;
\ell = 1,2,
\;\;
\text{and}
\;\;
\| F_{j}^{3} \|_{h^1}
\approx
\| F_{j}^{3} \|_{L^1}
\lesssim
2^{j\frac{ns}{2}},
\quad
j \in \N.
\end{align*}
From the choice of the functions $\theta$ and $\phi$,
we see that 
\begin{align*}
&
e^{i|D|^{s}} \theta(2^{-j} D) F^{1}_j(x)
=
\Big(
\psi_{1}(2^{-j}\xi)
\Big)^{\vee} (x),
\\
&
e^{i|D|^{s}} \phi(2^{-j} D) F^{2}_j(x)
=
\Big(
\psi_{2}(2^{-j}\xi)
\Big)^{\vee} (x).
\end{align*}
Now,
letting $(f,g,h) = (F_{j}^{1}, F_{j}^{2}, F_{j}^{3})$
in \eqref{dualform},
we can express the left hand side as
\begin{align*}
&
2^{jm}
\Big|
\int_{\R^n}
e^{i|D|^{s}} \theta(2^{-j} D) F^{1}_j(x) \cdot
e^{i|D|^{s}} \phi(2^{-j} D) F^{2}_j(x) \cdot
e^{i|D|^{s}} F^{3}_j(x) 
\,dx
\Big|
\\
&= 
2^{jm}
\int_{\R^n}
\Big(
\psi_{1}(2^{-j}\xi)
\Big)^{\vee} (x) \,
\Big(
\psi_{2}(2^{-j}\xi)
\Big)^{\vee} (x) \,
\Big(
\psi_{3}(2^{-j}\xi)
\Big)^{\vee} (x) \,
\,dx
\\
&
=2^{jm} 2^{2jn}
\iint_{\R^{2n}}
\psi_{1} (\xi+\eta) \,
\psi_{2} (\xi) \,
\psi_{3} (\eta)
\,d\xi d\eta
=: (\ast) ,
\end{align*}
which is estimated from below as
\begin{align*}
(\ast)
\ge 2^{jm} 2^{2jn}
\iint_{ \substack{ 1-2^{-10} \le |\xi| \le 1+2^{-10} \\ \frac{3}{2} -2^{-10} \le |\eta| \le \frac{3}{2} +2^{-10} } }
\psi_{1} (\xi+\eta) \,
\psi_{2} (\xi) \,
\psi_{3} (\eta)
\,d\xi d\eta
= A \, 2^{jm} 2^{2jn}
\end{align*}
for some constant $A$ depending only on $n$.
Therefore,
we have 
\[
2^{jm} 
2^{2jn} 
\lesssim
2^{j(n-\frac{ns}{2})}\;
2^{j(n-\frac{ns}{2})}\;
2^{j\frac{ns}{2}},
\]
which concludes that $m \le -ns/2$.


{\it Proof of the case $2 < s < \infty$.}
Take radial functions 
$\psi, \varphi \in C^\infty_0(\R^n)$
such that 
\begin{align*}
(\psi)^{\vee} \;\text{is real-valued}, 
\quad
&
| (\psi)^{\vee} | \ge 1 \;\;\text{on}\;\;
\{ |x| \le \delta\}, 
\quad
\supp \psi \subset \{ 1/2 \le |x| \le 2 \},
\\
( \varphi )^{\vee} \;\text{is non-negative}, 
\quad
&
| (\varphi )^{\vee} | \ge 1 \;\;\text{on}\;\;
\{ |x| \le \delta\}, 
\quad
\supp \varphi \subset \{ |x| \le 10 \},
\end{align*}
for some constant $\delta > 0$,
and set
\begin{align*}
 F_{j} (x) 
=
\Big(
e^{-i|\xi|^s}
\psi(2^{-j}\xi)
\Big)^{\vee} (x) ,
\quad
 F_{0} (x) 
=
\Big(
e^{-i|\xi|^s}
\varphi(\xi)
\Big)^{\vee} (x) .
\end{align*}
Then, 
we again have 
\begin{align*}
\| F_{j} \|_{L^\infty}
\lesssim
2^{j(n-\frac{ns}{2})}
\quad\text{and}\quad
\| F_{0} \|_{h^1}
\lesssim 1,
\end{align*}
and 
\begin{align*}
e^{i|D|^{s}} \theta(2^{-j} D) F_j(x)
=
e^{i|D|^{s}} \phi(2^{-j} D) F_j(x)
=
\Big(
\psi(2^{-j}\xi)
\Big)^{\vee} (x)
=
2^{jn} (\psi)^{\vee}(2^j x).
\end{align*}
Now,
set $(f,g,h) = (F_{j}, F_{j}, F_{0})$ in \eqref{dualform}.
Then  the left-hand side is estimated from below as 
\begin{align*}
&
2^{jm}
\Big|
\int_{\R^n}
e^{i|D|^{s}} \theta(2^{-j} D) F_j(x) \cdot
e^{i|D|^{s}} \phi(2^{-j} D) F_j(x) \cdot
e^{i|D|^{s}} F_0(x) 
\,dx
\Big|
\\
&= 
2^{jm}
2^{2jn}
\int_{\R^n}
\big( ( \psi )^{\vee} (2^{j} x) \big)^{2} \,
( \varphi )^{\vee} (x) 
\,dx
\ge 
2^{jm}
2^{2jn}
\int_{|2^{j} x| \le \delta}
\big( ( \psi )^{\vee} (2^{j} x) \big)^{2} \,
( \varphi )^{\vee} (x) 
\,dx
\ge 
A 2^{jm} 2^{jn}
\end{align*}
for some constant $A$ depending only on $n$ and $\delta$.
Therefore,
we conclude that
\[
2^{jm}
2^{jn}  
\lesssim
2^{j(n-\frac{ns}{2})} \;
2^{j(n-\frac{ns}{2})} ,
\]
which yields that $m \le - n(s-1)$.
\end{proof}

\begin{rem}
In the same way as above (or by duality),
we see that the order $m_s$ is also sharp for
$h^1 \times L^\infty \to L^1$
and 
$L^\infty \times h^1 \to L^1$
boundedness.
\end{rem}

\section*{Appendix}
In this appendix, we recall the atomic decomposition of $L^1$.
We say that a function $h \in L^1$ is an atom 
if there exist $x_0 \in \R^n$ and $0 < r \le 1$ such that 
\begin{align*}
\supp h \subset B(x_0, r),
\quad
\|h\|_{L^\infty} \le r^{-n}.
\end{align*}
Then we prove that 
for every $h \in L^1$ 
there exist 
a sequence of atoms $\{h_j\}$ and
a sequence of nonnegative real numbers $\{\lambda_j\}$
such that 
\[
h = \sum_{j} \lambda_j h_j,
\quad
\|h\|_{L^1} \approx \sum_{j} \lambda_j.
\]

Suppose $h \in L^{1}$. 
We define the sets $E_{k}$, $k \in \Z$, by 
\[
E_{k} = \{ x \in \R^n \mid 2^{k} \le |h(x)| < 2^{k+1} \} .
\]
For each $k \in \Z$, we choose a family 
of disjoint cubes 
$\{Q_{k, \nu}\}_{\nu}$ such that 
\[
\diam (Q_{k, \nu})\le 1, 
\quad 
E_{k} 
\subset 
\bigcup_{\nu} 
Q_{k, \nu}, 
\quad\text{and}\quad
|E_{k}| 
\le 
\sum_{\nu} 
\big| Q_{k, \nu} \big| 
\le 2 |E_{k}|.
\]
Let  
$c_{k,\nu}$ be the center of $Q_{k,\nu}$ and 
$r_{k,\nu}=\diam (Q_{k,\nu})$. 
Since $\{ E_{k} \}_{k}$ and $\{ Q_{k, \nu} \}_{\nu}$ 
are disjoint families,
we can write $h$ as
\[
h(x) = \sum_{k \in \Z} h(x) \, \ichi_{E_{k}}(x)
= \sum_{k \in \Z} \sum_{\nu \in \Z^n} h(x) \, \ichi_{E_{k}\cap Q_{k,\nu}}(x)
= \sum_{k \in \Z} \sum_{\nu \in \Z^n} \lambda_{k,\nu} \, h_{k,\nu}(x)
\]
with 
\[
\lambda_{k,\nu}
=
2^{k+1} r_{k,\nu}^{n},
\quad
h_{k,\nu}(x) = 
\frac{h(x) \, \ichi_{E_{k}\cap Q_{k,\nu}}(x) }
{2^{k+1} r_{k,\nu}^n }.
\]
Then, we see that
\[
\| h_{k, \nu} \|_{L^{\infty}} \le r_{k, \nu}^{-n},
\quad
\supp h_{k, \nu} 
\subset Q_{k, \nu} 
\subset B(c_{k,\nu}, r_{k,\nu}), 
\]
and thus each $h_{k,\nu}$ is an atom, 
and also 
\[
\sum_{k \in \Z} \sum_{\nu \in \Z^n} \lambda_{k,\nu}
=
\sum_{k \in \Z} \sum_{\nu \in \Z^n} 2^{k+1} r_{k,\nu}^n  
\approx 
\sum_{k \in \Z} \sum_{\nu \in \Z^n} 2^{k} |Q_{k, \nu}| 
\approx 
\sum_{k \in \Z} 2^{k} |E_{k}| 
\approx 
\| h \|_{L^1}. 
\]
This completes the proof.

\end{document}